\documentclass[10pt]{amsart}

\usepackage{amsmath,amsthm,amssymb}

\usepackage{a4wide}
\newtheorem{thm}{Theorem}[section]
\newtheorem{cor}[thm]{Corollary}

\newtheorem{defn}[thm]{Definition}

\newtheorem{theorem}{Theorem}[section]

\newtheorem{lemma}[theorem]{Lemma}

\newtheorem{proposition}[theorem]{Proposition}

\newtheorem{example}[theorem]{Example}

\def\R{\mathbb R}

\def\g{\mathfrak g}

\def\R{\mathbb R}

\def\g{\mathfrak{g}}

\def\R{\mathbb{R}}

\def\R{{\mathbb R}}

\def\L{{\cal   L}}

\def\hb#1{\hbox{#1}}

\def\exp#1{\hb{exp}(#1)}

\def\hb #1{\hbox{#1}}

\def\hb#1{\hbox{#1}}

\def\dim#1{\hb{dim}(#1)}

\def\L1#1{L^1(#1)}

\begin{document}


\title{Note on the cortex of two-step nilpotent Lie algebras }

\address[1]{ King Saud University, college of science, Department of
mathematics,\\
Riyadh, P.O Box 2455, Riyadh 11451, K.S.A.}

\author[B. Dali]{B\'echir Dali}

\email{bdali@ksu.edu.sa}



\date{\today}

\maketitle

\begin{abstract}
In this paper, we construct an example of a family of
$4d$-dimensional two-step nilpotent Lie algebras $(\frak g_d)_{d\geq
2}$ so that the cortex of the dual of each $\frak g_d$ is a
projective algebraic set. More precisely, we show that the cortex of
each dual $\frak g_d^*$ of $\frak g_d$ is the zero set of a
homogeneous polynomial of degree $d$. This example is a
generalization of one given in "Irreducible representations of
locally compact groups that cannot be Hausdorff separated from the
identity representation" by "{\sc M.E.B. Bekka, and E. Kaniuth}".
\end{abstract}
\

\vskip 0.5 cm

\section{Introduction}

\vskip 0.5 cm The cortex of general locally compact group $G$ was
defined in \cite{Vershik} as
$$ cor (G)=\{\pi \in\widehat G,   \textrm{ which cannot be Hausdorff-separated from the
identity representation}   { \bf 1}_G\}, $$ where $\widehat G$ is
the dual of $G$ (set of unitary irreducible representations of $G$),
that is, $\pi\in cor (G)$ if and only if for all neighborhood $V$ of
${\bf 1}_G$ and for each neighborhood $U$ of $\pi$, one has $V \cap
U$ is non-empty set. Note that $\widehat G$ is equipped with the
topology of Fell which can be described in terms of weak containment
(see \cite {Fell}) and, in general, is not separated. However, if
$G$ is abelian, then $\widehat G$ is separated and hence $
cor(G)=\{{\bf 1}_G\}$. The set

When $G$ is a connected and simply connected nilpotent Lie group
with Lie algebra $\frak g$, the Kirillov theory says that $\frak
g^*/{Ad^*(G)}$ and $\widehat G$ are homomorphic, where $Ad^*(G)$
denotes the coadjoint representation of $G$ on the dual $\frak g^*$
of $\frak g$. Hence, for this class of Lie groups, $cor(G)$ can be
identified with certain $Ad^*(G)$-invariant subset of $\frak g^*$.
From \cite{Bekka}, one introduces the cortex of $\frak g^*$ as
$$
Cor(\g^*)=\{\ell=\lim_{m\rightarrow\infty}Ad^*_{s_m}(\ell_m),
\textrm  { where }
 \{s_m\}\subset G  \textrm  { and } \{\ell_m\}\subset\mathfrak g^*
\textrm  {such that } \lim_{m\rightarrow\infty}\ell_m=0 \}
$$
and we have $\pi_\ell\in cor (G)$ if and only if $\ell\in Cor(\frak
g^*)$. Note that in the case of general Lie groups, the two
definitions are not so easily related. With the above in mind, it
comes out that the set in question can be defined on $\frak
g^*/Ad^*G$ as in \cite{Boidol} and one can investigate the
parametrization of $\frak g^*$ to determine such a set.

Motivated by this situation, the authors in \cite{Boidol} define the
cortex $C_V(G)$ of a representation of a locally compact group $G$
on a finite-dimensional vector space $V$ as the set of all $v\in V$
for which $G.v$ and $\{0\}$ cannot be Hausdorff-separated in the
orbit-space $V/G$. They give a precise description of $C_V(G)$ in
the case $G=\R$. Moreover, they consider the subset $IC_V(G)$ of $V$
consisting of the common zeroes of all $G$-invariant polynomials $p$
on $V$ with $p(0)=0$. They show that $IC_V(G)=C_V(G)$ when $G$ is a
nilpotent Lie group of the form $G=\R\ltimes\R^n$ and $V=\frak g^*$
the dual of the Lie algebra $\frak g$. This fails for a general
nilpotent Lie group, even in the case of two-step nilpotent Lie
group, in \cite{Bekka}, the authors show that the cortex of any
two-step nilpotent Lie algebra $\g$ is the closure of the set:
$$
\{ad^*_X(\ell),\,X\in \mathfrak{g},\, \ell\in
{\mathfrak{g}}^*\,\}\qquad \qquad (1).
$$
and they give a counter-example of $8$-dimensional Lie algebra
$\frak g$ for which $Cor(\frak g^*)\subsetneq ICor(\frak g^*)$. In
\cite{Bak1}, the author gives an explicit description of $Cor(\frak
g^*)$ when $\frak g$ is a nilpoten Lie algebra with $\dim{\frak
g}\leq 6$ by the means of parametrization of the coadjoint orbits in
$\frak g^*$. In \cite{Dali} one gives an explicit description of the
cortex of ceratin class of exponential Lie algebras $\frak g$ by
investigating the tools of parametrization of the coadjoint orbits
in the dual space $\frak g^*$ (see for example \cite{ABCD, ACD}).

Fixing the class of two-step nilpotent Lie algebras, for any $\ell$
belonging to the dual space of a such Lie algebra, its coadjoint
orbit $\mathcal O$ satisfies
$$ \mathcal O=\{\ell\}+T_\ell\mathcal O,
$$ where $T_\ell\mathcal O$ is the tangent space to the orbit $\mathcal O$ at $\ell$.
From this, we deduce that each coadjoint orbit is a flat (affine)
symplectic manifold. Combining this with the property (1), we
intuitively think that the cortex of such class of Lie algebras
needs not to be far from affine sets however in this paper, we show
that for this class of Lie algebras the cortex can be very far from
affine or flat sets. Indeed, we give a generalization of the example
given in \cite{Bekka} pp. $210$, our example is a family of
$4d$-dimensional two-step nilpotent Lie algebras $(\frak g_d)_{d\geq
2}$ such that the cortex of each $\frak g_d^*$ is the zero set of a
homogeneous polynomial of degree $d$ in the complement $\frak
z_d^\perp$ of the center $\frak z_d$ of $\frak g_d$.

The paper is organized as follows, in Section 2, we give the
definition of the main tools of cortex and we recall some essential
results. In section 3, we give the main example and we conclude the
paper by some conjecture about the optimality of the family $(\frak
g_d)_{d\geq 2}$.

\vskip 0.5cm

\section{Preliminaries}

\subsection{ Definitions and notations}

\vskip 0.5cm

We begin by setting some notations and useful facts which will be
used throughout the paper. This material is quite standard.
Throughout, $G$ will always denote an $n$-dimensional connected and
simply connected two-step nilpotent Lie group with (real) Lie
algebra $\frak g$ so that $[\g,[\g,\g]]=0$. We denote by $\frak z$
the center of $\frak g$ and $\g^*$ denotes the dual of $\g$. Let us
denote $\frak z^\perp$ to be
$$
\frak z^\perp=\{\ell\in\frak g^*: \ell(Z)=0\quad \forall Z\in\frak
z\}. $$ $G$ acts on $\frak g$ by the adjoint action denoted by $Ad$
and on $\frak g^*$ by the coadjoint action denoted by $Ad^*$. Since
$G$ is nilpotent then the exponential mapping from $\frak g$ onto
$G$ is a global diffeomorphism, and we can write $Ad_{\exp X}=e^{ad
X}$ where $adX(Y)=[X,Y]$. Since $\frak g$ is two-step nilpotent Lie
algebra, one has $e^{ad X}=Id_{\frak g}+adX$, and hence the
coadjoint action of $G$ on the dual space $\frak g^*$ of $\frak g$
is given by $Ad^*_{\exp X}=Id_{\frak g^*}+ad^*X$ where $ad^*$ is the
coadjoint action of $\frak g$ on $\frak g^*$. As a consequence, for
any $\ell\in\frak g^*$, if ${\mathcal O}=Ad^*(G)\ell$ denotes the
coadjoint orbit of $\ell$, then
$$
{\mathcal O}=\{\ell\}+T_\ell\mathcal O,
$$
where $T_\ell\mathcal O$ is the tangent space of ${\mathcal O}$ at
$\ell$. Thus we see that the coadjoint orbits in two-step nilpotent
Lie algebras are flat symplectic manifolds.

Following \cite{Boidol}, we recall the following:
\begin{defn}\label{df1} Let $\ell\in\mathfrak{g}^*  $.  We define the cortex of $\frak g^*$ as
$$
Cor(\mathfrak g^*)=\{\lim_{m\rightarrow\infty} Ad^*_{s_m}(\ell_m)~|~
 (s_m)_m\subset G,~(\ell_m)_m\subset\frak g^* \textrm { with }
 \lim_{m\rightarrow\infty} \ell_m=0\}
$$
\end{defn}
As a consequence of this definition, since $\frak g$ is two-step
nilpotent Lie algebra, then we can write
$$
Cor(\mathfrak g^*)=\{\lim_{\ell\rightarrow
0}ad^*_{X_n}(\ell_n),\quad (X_n)_n\subset\frak g, (\ell_n)_n\subset
g^*\}=\lim_{\ell\rightarrow 0}T_\ell\mathcal O.$$
 In \cite{Bekka} there is a study of the cortex of several
locally compact groups and especially the case of connected Lie
groups. In particular, when $G$ is a connected two-step nilpotent
Lie group, one has the following:
\begin{proposition} \label{2-step} Let $\frak g$ be a nilpotent Lie
algebra of class $2$ (i.e, $[\frak g,[\frak g,\frak g]]=0$), and let
$G=\exp{\frak g}$ be the associated Lie group. Denote by $ad^*$ the
coadjoint representation of $\frak g$ on $\frak g^*$. Let $f\in\frak
g^*$. Then the corresponding representation $\pi_f$ of $G$ belongs
to $cor(G)$ if and only if $f$ belongs to the closure of the subset
$\{ ad_X^*(\ell), X\in\frak g,\ell\in\frak g^*\}$ of $\frak g^*$.
\end{proposition}
In \cite{Dali}, the author shows the following:
\begin{lemma}\label{cortex3}
If  $\g$  is a two-step nilpotent Lie algebra, and  if  the
coadjoint orbits have codimension 0 or 1 in $\frak z^\perp$, then
$$
Cor(\g^*)=\mathfrak z^\perp.
$$
\end{lemma}

 Since the coadjoint orbits in two-step
nilpotent Lie algebras are affine, we naturally think that the
cortex of any two-step nilpotent Lie algebra will be an affine
subset of $\frak z^\perp$. Again, in \cite{Bekka} it was given an
interesting example of a two-step nilpotent Lie algebra $\frak g$ so
that the corresponding cortex in $\frak g^*$ is a projective
algebraic set given by a quadric, let us recall it.
\begin{example}
Let $\frak g$ be the Lie algebra of dimension $8$ with basis
$(X_1,\dots X_6,Z_1,Z_2)$ and nontrivial commutators
$$
[X_1,X_5]=[X_2,X_3]=Z_1, [X_1,X_6]=[X_2,X_4]=Z_2.$$ Then $\frak g$
is nilpotent of class $2$, and it is easily verified that the center
$\frak z$ of $\frak g$ equals $\mathbb R Z_1+\mathbb R Z_2$. Let
$(X_1^*,\dots, X_6^*, Z_1^*, Z_2^*)$ denote the corresponding basis
of $\frak g^*$, and let $G=\exp{\frak g}$. It follows by Proposition
\ref{2-step} that $$ cor(G)=\{\pi_f\in\widehat G;
f=\sum_{i=1}^6t_iX_i^* \text { with } t_3t_6=t_4t_5\},
$$
or equivalently,
$$
Cor(\frak g^*)=\{f\in\frak z^\perp : f=\sum_{i=1}^6t_iX_i^* \text {
with } t_3t_6=t_4t_5\}.
$$
\end{example}
In this example, it was shown that the set of invariant polynomials
on $\frak g^*$ is
$$
Pol(\frak g^*)^G=\mathbb R[z_1,z_2, z_1x_4-z_2x_3,z_1x_6-z_2x_5],
$$
and
$$
ICor(\frak g^*)=\{\ell\in\frak g^*: P(\ell)=P(0),\quad\forall P\in
Pol(\frak g^*)^G\}=\frak z^\perp.
$$
Then $Cor(\frak g^*)\varsubsetneq ICor(\frak g^*)$. Thus the cortex
of $G$ is strictly contained in the set of all irreducible
representations of $G$ having a trivial infinitesimal charcter.

Up to writing this note, I don't see any example of two-step
nilpotent Lie algebra $\frak g$ so that that $Cor(\frak g^*)$ is the
zero set of polynomials whose degree is grater then $2$. In the
sense that we think $Cor(\frak g^*)$ is the zero set of polynomials
whose degree is equal of less then $2$.

\section{Main example}

 Let $d$ be an integer with
$d\geq 2$ and $\frak g_d$ be the Lie algebra whose Jordan-H\"{o}der
basis
$$\mathcal B=(Z_1,\dots, Z_d, Y_1,Y_2,\dots,
Y_{2d-1},Y_{2d},X_1,\dots, X_d),
$$ and nontrivial brackets
$$
[X_i,Y_{2i-1}]=Z_1, i=1,\dots,d, [X_k,Y_{2k}]=Z_{k+1},
k=1,\dots,d-1, [X_{2d},Y_{2d}]=Z_2+\dots+Z_d.
$$
\begin{proposition}\label{parametrization}
For each Lie algebra $\frak g_d$ ($d\geq 2$), one has

\begin{itemize}

\item[(i).] The generic orbits are $2d$ dimensional affine manifold defined by the minimal layer
$$
\Omega_d=\{\ell\in\frak g_d^*: \ell(Z_1)\neq 0\}.
$$

\item[(ii).] Any coadjoint orbit $\mathcal O$ of $\ell\in\Omega_d$ is
given by
$$
\mathcal
O=G\cdot\ell=\{\prod_{i=1}^dAd^*_{\exp{t_iY_{2i-1}}}Ad^*_{\exp{s_iX_i}}(\ell),\quad
(t_1,\cdots,t_d,s_1,\cdots,s_d)\in\mathbb R^{2d}\}.
$$

\item[(iii).] The algebra of $G$-invariant polynomials is
$$
Pol(\frak g_d^*)^G=\mathbb
R[z_1,\cdots,z_d,z_1y_2-z_2y_1,\cdots,z_1y_{2d-2}-z_{d-1}y_{2d-3},z_1y_{2d}-(z_2+\cdots+z_d)y_{2d-1}].
$$
\end{itemize}
\end{proposition}

\begin{proof}
\begin{itemize}

\item[(i).] Fixing the Jordan H\"{o}lder basis $\mathcal B_d$ and denote
$\mathcal B_d=(U_1,\dots,U_{4d})$ and $\mathcal
B^*_d=(U_1^*,\dots,U_{4d}^*)$ its dual basis with
$$ U_i=\left\{
      \begin{array}{ll}
        Z_i, & \hbox{ if} 1\leq i\leq d,\\
        Y_{i-d}, & \hbox{if } d+1\leq i\leq 3d;\\
        X_{i-3d}, & \hbox{if }3d+1\leq i\leq 4d\\
              \end{array}
    \right.
$$
Using the methods of \cite{ABCD}, we can see that the minimal layer
in $\frak g^*_d$ is
$$
\Omega_d=\{\ell\in\frak g^*: \ell(U_1)=\ell(Z_1)\neq 0\}, $$ and it
corresponds to the set of jump indices ${\bf e_d}={\bf i_d}\cup{\bf
j_d}$ with ${\bf i_d}=\{d+1<d+3<\dots<3d-1\}$ and ${\bf
j_d}=\{3d+1,3d+2,\dots,4d\}$. The cross-section $\Sigma_d$ is given
by
$$
\Sigma_d=(\sum_{k\notin{\bf e}}\mathbb
RU_k^*)\cap\Omega=\left(\sum_{k=1}^d\mathbb RZ_k^*+\mathbb
RY_{2k}^*\right)\cap\Omega_d.
$$
Concerning (ii) and (iii), they can be easily showed by using also
the methods of \cite{ABCD} in the parametrization of coadjoint
orbits.

\end{itemize}

\end{proof}
\begin{theorem}
Let $(Z_1^*,\dots, Z_d^*,Y_1^*,\dots, Y_{2d}^*,X_1^*,\dots,X_d^*)$
be the corresponding dual basis in $\frak g_d^*$. The cortex of
$\frak g_d^*$ is the dual of the Lie algebra $\frak g_d$ is a
projective set. More precisely, if we denote
$\ell=\sum_{i=1}^d(z_iZ_i^*+x_iX_i^*)+\sum_{j=1}^{2d}y_jY_j^*\in\frak
g^*$ by $\ell=(z_i,y_j,x_k)$ then $Cor(\frak g_d^*)$ is the
projective algebraic set given by
$$
Cor(\frak g_d^*)=\{\ell=(z_i,y_j,x_k):~z_1=\dots=z_d=
y_{2d-1}(\sum_{i=1}^{d-1}y_{2i}\prod_{j=1, j\neq
i}^{d-1}y_{2j-1})-y_{2d}\prod_{j=1}^{d-1}y_{2j-1}=0\}.
$$
\end{theorem}
\begin{proof}
Note that since $\frak g_d$ is exponential Lie algebra and
$\Omega_d$ is dense in $\frak g_d^*$ (Zariski open set)then
$$
Cor(\frak g_d^*)=\{\lim_mAd^*_{\exp{X_m}}\ell_m,~ (\ell_m)\in\frak
g_d^* ,~~ (\ell_m)_m\in\Omega_d, \text { and } \lim_m\ell_m=0\}.$$
On other hand each coadjoint orbit in $\Omega$ corresponds to the
set of jump indices ${\bf e_d}$ and satisfies
$$
\mathcal
O=G\cdot\ell=\{\prod_{i=1}^dAd^*_{\exp{t_iY_{2i-1}}}Ad^*_{\exp{s_iX_i}}(\ell),\quad
(t_1,\cdots,t_d,s_1,\cdots,s_d)\in\mathbb R^{2d}\}.
$$
Hence we can deduce that $Cor(\frak g_d^*)$ is the closure of the
set
$$\{ad^*_X(\ell),~(\ell)\in\Omega_d \text { and }
X\in Vect\{Y_{2k-1},X_k, ~~1\leq k\leq d\}\}.$$ Now let
$\ell=(z_i,y_j,x_k)\in\Omega_d$ and $\xi\in\mathcal O$, with
$\xi=\sum_{i=1}^d\lambda_iZ_i^*+\sum_{j=1}^{2d}\gamma_iY_i^*+\sum_{k=1}^d\beta_kX_k^*$,
then one has
$$\xi=\left\{
    \begin{array}{ll}
      \lambda_i &= z_i, \hbox { if }  i=1,\cdots,d; \\
      \gamma_{2j-1}&=y_{2j-1}-s_jz_1, \hbox { if } j=1,\cdots, d-1 ; \\
      \gamma_{2j} &=y_{2j}-s_jz_{j+1}, \hbox { if } j=1,\cdots, d-1; \\
      \gamma_{2d-1}&=y_{2d-1}-s_dz_1; \\
       \gamma_{2d}&=y_{2d}-s_d(z_2+\cdots+z_d); \\
      \beta_k&= x_k+t_kz_1, \hbox { if } k=1,\cdots,d.
    \end{array}
  \right.
$$
On other hand, one can see that the tangent space to $\mathcal O$ at
$\ell$ is
$$ T_\ell\mathcal O=\{ad^*_X(\ell), X\in
VectVect\{Y_{2k-1},X_k, ~~1\leq k\leq d\}\},$$ and hence any element
in this space has coordinates
$$
\left\{
    \begin{array}{ll}
      \lambda_i= & 0, \hbox { if }  i=1,\cdots,d; \\
      \gamma_{2j-1}&=-s_jz_1, \hbox { if } j=1,\cdots, d-1 ; \\
      \gamma_{2j} &=-s_jz_{j+1}, \hbox { if } j=1,\cdots, d-1;; \\
      \gamma_{2d-1}&=-s_dz_1; \\
       \gamma_{2d}&=-s_d(z_2+\cdots+z_d); \\
      \beta_k&= t_kz_1, \hbox { if } k=1,\cdots,d.
    \end{array}
  \right.
$$
Then we conclude
$$ Cor(\frak g_d^*)=\{\ell=(z_i,y_j,x_k)\in\frak
g_d^*: z_i=0,y_{2d-1}(\sum_{i=1}^{d-1}y_{2i}\prod_{j=1, j\neq
i}^{d-1}y_{2j-1})-y_{2d}\prod_{j=1}^{d-1}y_{2j-1}=0\}.
$$
\end{proof}

\begin{cor}
For each integer $d\geq2$ let $\frak z_d$ denotes the center of the
Lie algebra $\frak g_d$. Then for any $d\geq 2$ one has
$$
Cor(\frak g_d^*)\subsetneq ICor(\frak g_d^*)=\frak z^\perp.
$$
\end{cor}
\subsection{Concluding remarks}

\begin{itemize}

\item[1.] Let us remark that the cross-section mapping $P_d:\Omega_d\rightarrow\Sigma_d$ is as follows
$$
P_d:\ell=(z_i,y_j,x_k)  \mapsto
P(\ell)=\sum_{i=1}^dz_iZ_i^*+\sum_{i=1}^{d-1}(y_{2i}-\frac{z_{i+1}}{z_i}y_{2i-1})Y_{2i}^*+
(y_{2d}-\frac{z_1+\dots+z_d}{z_1})Y_{2d}^*.
$$

\item[2.] I believe that the family of Lie algebras
$(\g_d)_{d\geq 2}$ is optimal in the sense that if for some
$k$-dimensional two-step nilpotent Lie algebra $\frak g$ the cortex
of the dual of $\frak g^*$ satisfies
$$ Cor(\frak
g^*)=\{\ell\in\frak z^\perp,~ P_i(\ell)=0, \text { with each  } P_i
\text { is a homogeneous polynomial of degree } k\}, $$ then $\dim
\frak z\geq k$ and $\dim{\frak g}\geq 4k$ where $\frak z$ is the
center of $\frak g$.

\end{itemize}

\end{document}